\documentclass{article}
\usepackage[utf8]{inputenc}
\usepackage{amsmath,amssymb,amsthm}
\usepackage{hyperref}
\usepackage[capitalize,noabbrev]{cleveref}
\usepackage{graphicx,float}
\usepackage{bbm}
\usepackage{xcolor}
\title{Dimension drop of harmonic measure for some finite range random walks on Fuchsian Schottky groups}
\author{Ernesto García \and Pablo Lessa}

\newtheorem{maintheorem}{Theorem}

\newtheorem{theorem}{Theorem}
\newtheorem{lemma}{Lemma}

\newtheorem{proposition}{Proposition}
\newtheorem{conjecture}{Conjecture}

\newtheorem{example}{Example}
\newtheorem{definition}{Definition}

\DeclareMathOperator{\dist}{dist}

\newcommand{\Pof}[1]{\mathbb{P}\left(#1\right)}

\newcommand{\R}{\mathbb{R}}
\renewcommand{\H}{\mathbb{H}}
\newcommand{\D}{\mathbb{D}}
\newcommand{\C}{\mathbb{C}}
\renewcommand{\P}{\mathbb{P}}

\newcommand{\Z}{\mathbb{Z}}

\DeclareMathOperator{\supp}{supp}

\DeclareMathOperator{\prefix}{Prefix}
\DeclareMathOperator{\SL}{SL}

\begin{document}
\maketitle

\begin{abstract}
 We prove that the harmonic measures of certain finite range random walks on Fuchsian Schottky groups, have dimension strictly smaller than the Hausdorff dimension of the corresponding limit set.
\end{abstract}

\section{Introduction}

Since early work of Furstenberg (see \cite{furstenberg} and \cite{ledrappierpositive}) it is known that the norm of a typical product of i.i.d. random matrices in \(\SL_2(\R)\) will grow exponentially as the number of factors increases under mild assumptions on the common distribution \(\mu\) of the matrices. 

Furthermore, by the Oseledets' theorem \cite{oseledets}, there is almost surely a well-defined unique one-dimensional subspace of vectors whose norm does not grow at the same rate as the matrix product does.    

These two facts can be interpreted, using the action of these matrices on the hyperbolic plane \(\H^2\), as stating that typical random walks have a well-defined asymptotic speed and converge to a particular direction in the ideal boundary (see \cite{kaimanovich} and \cite{karlsson-margulis}).

In this article, we are interested in the distribution \(\nu\) of this point on the ideal boundary, which is the unique \(\mu\)-stationary measure on the boundary.   

A well known conjecture (we will review the literature below) is that when \(\mu\) has finite support, and this support generates a discrete group, the stationary measure \(\nu\) should be singular with respect to Lebesgue measure.

One of the points of the present work is to extend the range of the conjecture by allowing stationary measures whose support is known a-priori to have dimension strictly less than one, and therefore are known to be singular.  We conjecture, for \(\mu\) with finite support generating any non-elementary discrete group, that the stationary measure \(\nu\) (which is known to be exact dimensional, see below) has dimension strictly smaller than that of its support.

We give some evidence for the conjecture by proving it for finite support measures on two generator Schottky groups satisfying a (rather strong) additional hypothesis.   We believe that the same methods could yield an analogous result for Schottky groups in \(d \ge 2\) generators.  However, the removal of the condition on the finite support measure needed for our proof seems to require new ideas.

 \subsection{A dimension drop conjecture}
 
To make the statements above precise, let \(G\) be the group of isometries of the hyperbolic plane \(\H^2\) and \(\Gamma < G\) be a non-elementary discrete subgroup (we refer to \cite{anderson} and \cite{dalbo} for basic hyperbolic geometry).

Let \(\Lambda\) be the limit set of \(\Gamma\), which is a closed \(\Gamma\) invariant subset of the visual boundary \(\partial \H^2\) of \(\H^2\).   We denote by \(\dim(\Lambda)\) the Hausdorff dimension of \(\Lambda\) with respect to the visual distance based at any point \(o \in \H^2\).

For any finitely supported probability measure \(\mu\) whose support generates \(\Gamma\) as a semi-group there exists a unique \(\mu\)-stationary probability \(\nu\) on \(\partial \H^2\).  That is, \(\nu\) satisfies
\[\nu = \sum\limits_{g \in \Gamma}\mu(g)g_*\nu,\]
where \(g_*\nu\) denotes the push-forward of \(\nu\) by \(g\).   Since any \(\Gamma\)-invariant compact set admits a \(\mu\)-stationary measure, it follows by uniqueness that the support of \(\nu\) must be contained in \(\Lambda\).

The measure \(\nu\) is also the harmonic measure of any random walk of the form \(x_n = g_1\cdots g_n o\) where \(o \in \H^2\) and \(g_1,\ldots, g_n,\ldots\) are i.i.d. random elements with common distribution \(\mu\).  By this we mean that it is the distribution of the random limit point \(x_\infty = \lim\limits_{n \to +\infty} x_n\).

Furthermore, \(\nu\) is exact dimensional (see \cite{tanaka}, and also \cite{leprince}, \cite{hochman-solomyak}, and \cite{ledrappier-dimension}), i.e. there exists a non-negative real number \(\dim(\nu)\) such that
\[\dim(\nu) = \lim\limits_{r \downarrow 0}\frac{\log (\nu(B(\xi,r)))}{\log(r)},\]
for \(\nu\)-a.e. \(\xi \in \partial\H^2\), where \(B(\xi,r)\) denotes the ball of radius \(r\) centered at \(\xi\) with respect to the visual metric based at some point \(o \in \H^2\).  

The dimension \(\dim(\nu)\) can be characterized as the infimum of the Hausdorff dimensions of Borel sets with full \(\nu\)-measure (see for example \cite[Proposition 2.1]{young}).  Since \(\nu\) is supported on \(\Lambda\) one has \(\dim(\nu) \le \dim(\Lambda)\).  We conjecture that equality is never attained for finitely supported \(\mu\).

\begin{conjecture}[Dimension drop conjecture]\label{mainconjecture}
Let \(\mu\) be a finitely supported probability on \(G\) whose support generates (as a semi-group) a discrete non-elementary subgroup \(\Gamma\), and let \(\nu\) be the unique \(\mu\)-stationary measure on the limit set \(\Lambda\) of \(\Gamma\).   Then \(\dim(\nu) < \dim(\Lambda)\).
\end{conjecture}

The conjecture is motivated by the observation that, in other contexts, the harmonic measure of a random walk `built from local information' has a dimension drop with respect to the natural geometric measure on the boundary.

Two contexts where analogous statements have been proven are the harmonic measure of a Jordan domain bounded by a curve with Hausdorff dimension strictly larger than \(1\) (see \cite{makarov}), and the harmonic measure for the simple random walk on a Galton-Watson tree (see \cite{lyons-pemantle-peres}).

Another motivation for Conjecture \ref{mainconjecture} is to extend the scope of the following well known conjecture of singularity with respect to visual measure on the boundary (see \cite[Conjecture 1.21]{dkn2009}, and \cite{kl2011}) to the case where \(\Lambda\) is not the entire visual boundary, allowing it to be investigated in a larger family of examples.

\begin{conjecture}[Singularity conjecture]\label{klconjecture}
If \(\mu\) is a finite support probability whose support generates (as a semi-group) a Fuchsian group \(\Gamma\) then the unique \(\mu\) stationary measure \(\nu\) on \(\partial \H^2\) is singular with respect to the class of visual measures.
\end{conjecture}

We notice that Conjecture \ref{klconjecture} has been studied in other contexts such as, higher dimensional hyperbolic space \cite{randecker-tiozzo}, Teichmüller space (see \cite{gadre} and \cite{gadre-maher-tiozzo-teich}), and higher rank semi-simple Lie groups \cite{kl2011}.  In all these contexts analogous statements to Conjecture \ref{mainconjecture} seem to warrant further investigation.

Conjecture \ref{klconjecture} is known to be true when \(\Gamma\) is not co-compact but has finite co-volume (see \cite{guivarch-lejan}, \cite{bhm}, \cite{dkn}, \cite{gadre-maher-tiozzo}).  

Some progress has been made in the case when \(\Gamma\) is co-compact.  In \cite{carrasco-lessa-paquette} and \cite{kosenko} the conjecture is verified for nearest neighbor random walks on tilings by regular polygons.  A general result for symmetric random walks on the fundamental group of the surface of genus two is obtained in \cite{tk2020}).  For any finite support probability measure on the fundamental group of a closed surface, dimension drop for the harmonic measure associated to discrete and faithful representations of the group outside of a compact subset of Teichmüller space was established in \cite{aimpaper}.

Dimension drop of harmonic measure was established for finite range random walks on discrete groups \(\Gamma\) which are not virtually free and whose boundary is endowed with a distance coming from a word metric in \cite[Theorem 1.1]{gmm}.  However, their result does not apply in the current setting (see \cite[Remark 1.1]{gmm}).

If \(\mu\) is not assumed to be finitely supported then Conjecture \ref{mainconjecture} is false.   In the co-compact case a measure \(\mu\) whose stationary measure is continuous can be constructed via the Furtenberg-Lyons-Sullivan discretization procedure (see \cite{lyons-sullivan}) or by more general results of Connell and Muchnik \cite{connell-muchnik}.   In the convex co-compact case it has been shown that the probability can be chosen with finite exponential moment (see \cite[Appendix]{li}).

Conjecture \ref{klconjecture} is false if the support of \(\mu\) generates a dense group (see \cite{bourgain} and \cite{barany-pollicott-simon}).

Conjecture \ref{mainconjecture} should be contrasted with the classical case of self-similar sets defined by iterated function systems satisfying the open set condition.  In that case a probability \(\mu\) on the generators of the iterated function system can be found such that the \(\mu\)-stationary measure realizes the Hausdorff dimension of the self-similar set (see \cite{moran} and \cite{hutchinson}).

\subsection{A partial result for Schottky groups}

The purpose of this article is to prove some special cases of Conjecture \ref{mainconjecture} when the discrete group \(\Gamma\) is a Schottky group in two generators.

To state our result we will need to discuss the expression of elements of the group \(\Gamma\) in terms of words in a free generator.  For this purpose we fix the finite alphabet \(\Sigma = \lbrace a,b,a^{-1},b^{-1}\rbrace\) and let \(\Sigma^*\) denote the set of finite words on this alphabet, including the empty word which we denote by \(\varepsilon\).

We fix the rewriting rules \(R = \lbrace (aa^{-1},\varepsilon),(a^{-1}a,\varepsilon), (bb^{-1},\varepsilon),(b^{-1}b,\varepsilon)\rbrace\) which are the standard presentation for the free group with generators \(\lbrace a,b\rbrace\). Each word \(w \in \Sigma^*\) has a unique reduced form \(\overline{w}\) which is obtained by successively replacing any occurrence of the left-hand side of a pair in \(R\) with the empty word (though we will not need any deep results we refer to \cite{rewriting} as a general reference on term rewriting systems).

The set of reduced words \(\Sigma^*_R\) is a free group with generators \(a,b\) when endowed with the product \(w\cdot w' = \overline{ww'}\) (where \(ww'\) is the concatenation of \(w\) and \(w'\)). A \textit{prefix} of a word in $\Sigma^*_R$ is a subword obtained by deleting rightmost letters (see Section 2 for a precise definition).

We now consider a probability measure \(\mu\) on \(\Sigma^*_R\) whose support \(\supp(\mu)\) is finite and generates \(\Sigma^*_R\) as a semi-group. Denote the
set of prefixes of the words in \(\supp(\mu)\) by \(\prefix(\supp(\mu))\).

Suppose \(\rho:\Sigma^*_R \to \Gamma\) is an isomorphism from \(\Sigma^*_R\) to a Schottky group in \(G\). By this we mean that there exists family of half-planes \(\lbrace H_x: x \in \Sigma\rbrace\) whose closures in \(\H^2 \cup \partial \H^2\) are pairwise disjoint and such that \(\rho(x)H_y \subset H_x\) for all \(x\) and all \(y \in \Sigma \setminus \lbrace x^{-1}\rbrace\).  See Figure \ref{schottkyfigure}.

We let \(\nu\) denote the \(\rho_*\mu\)-stationary measure on the limit set \(\Lambda\) of \(\Gamma\). 

Our main theorem is the following:
\begin{maintheorem}[Dimension drop for some finite range random walks on Schottky groups]\label{maintheorem}
In the setting above suppose that \(a\) occurs as the last letter of exactly one element of \(\text{Prefix}(\supp(\mu))\). Then \(\dim(\nu) < \dim(\Lambda)\).
\end{maintheorem}

We remark that the techniques we use to prove Theorem \ref{maintheorem} could be used to obtain an analogous result on free groups of larger rank.  However, removing the hypothesis on the number of occurrences of the generator \(a\) in the support seems to require some new ideas.

The theorem applies in particular to nearest neighbor random walks, i.e. when \(\supp(\mu) = \Sigma\).  Examples with arbitrarily large cardinality may be obtained defining \(\supp(\mu)\) as the set of all non-empty reduced words with length at most \(n\), which either do not contain the letter \(a\), or are of the form \(aw\) where \(w\) does not contain \(a\).   For \(n = 2\) this yields
\[\supp(\mu) = \lbrace b,a^{-1},b^{-1},b^2, ba^{-1},a^{-1}b,a^{-2},a^{-1}b^{-1}, b^{-1}a^{-1}, b^{-2}\rbrace \cup \lbrace a,ab,ab^{-1}\rbrace.\]

\begin{figure}[H]
    \centering
    \includegraphics[width=\textwidth]{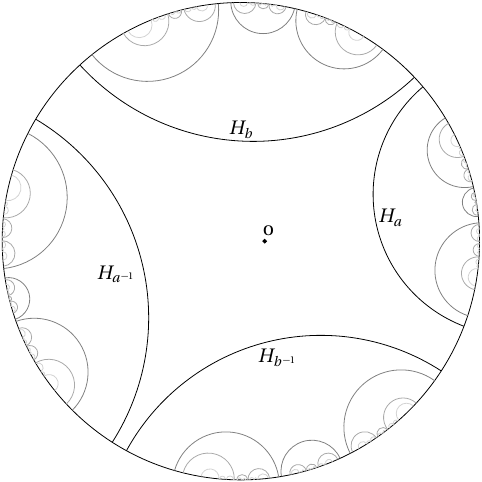}
    \caption{\label{schottkyfigure}Half-planes \(H_x\) for \(x \in \Sigma\) associated to the Schottky group \(\rho(\Sigma^*_R)\), in the Poincaré disk model of \(\H^2\).}
\end{figure}

\section{Hidden Markov property of harmonic measure}

We let \(\Sigma^\omega\) denote the set of infinite words.  We use \(w[i]\) for the \(i-\)th letter of a word (infinite or finite) so \(w = w[1]w[2]\cdots\), and set \(w[i:j] = w[i]w[i+1]\cdots w[j]\). For any $n\geq 1$, $w[1:n]$ is called a \textit{prefix} of $w$, and we also consider the empty word $\varepsilon$ and $w$ itself as prefixes. Let \(\Sigma^\omega_R\) denote the subset of infinite reduced words (i.e. every prefix is reduced).

The set \(\Sigma^\infty = \Sigma^* \cup  \Sigma^\omega\) is a compact Polish space under letter-wise convergence, i.e. \(w_n \to w\) if and only if for each \(i \in \lbrace 1,\ldots, |w|\rbrace\) one has \(w_n[i] = w[i]\) for all \(n\) large enough.  The subset \(\Sigma^\infty_R = \Sigma^*_R \cup \Sigma^\omega_R \subset \Sigma^\infty\) is compact.

In the setting of Theorem \ref{maintheorem} let \(g_1,g_2,\ldots\) be i.i.d. with common distribution \(\mu\). The limits
\[g_\infty = \lim\limits_{n \to +\infty}g_1g_2\cdots g_n,\]
and
\[\overline{g_\infty} = \lim\limits_{n \to +\infty}\overline{g_1g_2\cdots g_n},\]

exist almost surely and belong to \(\Sigma^\omega\) and \(\Sigma^\omega_R\) respectively.   We let \(\nu_{\Sigma^{\omega}}\) denote the distribution of \(g_\infty\) and \(\nu_{\Sigma^\omega_{R}}\) the distribution of \(\overline{g_\infty}\).

We define \(\pi:\Sigma^* \to \Sigma \cup \lbrace \varepsilon \rbrace\) as the projection associating to each non-empty word its last letter, and satisfying \(\pi(\varepsilon) = \varepsilon\).

\begin{theorem}[Reduced infinite words are hidden markov]\label{reducedhiddenmarkovtheorem}
  There exists a Markov chain \(x_1,x_2,\ldots\) on \(\prefix(\supp(\mu)) \setminus \lbrace \varepsilon\rbrace\) such that
 \[\lim\limits_{n \to +\infty}\pi(x_1)\pi(x_2)\cdots \pi(x_n),\]
 exists, belongs to \(\Sigma^\omega_R\) almost surely, and has distribution \(\nu_{\Sigma^\omega_R}\).
\end{theorem}

Theorem \ref{reducedhiddenmarkovtheorem} implies in particular that \(\nu_{\Sigma^\omega_R}\) is a hidden Markov measure.  This result follows from the recent work on coloured random walks \cite{bordenave-dubail} though our proof is different.   Since we need the specific statement above for what follows, we present the details in the present section.

We give an example illustrating that the harmonic measure is usually not Markov.

\begin{example}[Non Markovian harmonic measure]\label{example:non-Markovian}
    Let \(\mu\) be supported in the generator \(\{a^{-1},b^{-1},a^2b\}\), with \(\mu(\lbrace a^{-1}\rbrace) = \mu(\lbrace b^{-1}\rbrace) = \epsilon\) and \(\mu(\lbrace a^2b\rbrace) = 1-2\epsilon\), with \(\epsilon > 0\).

    We claim that
    \[\frac{\nu([ab])}{\nu([a])} < \frac{\nu([a^2b])}{\nu([a^2])},\]
    if \(\epsilon\) is small enough.

    For this purpose, we observe that if \(g_1,\ldots,g_n,\ldots\) are i.i.d.\ with distribution \(\mu\) and \(g_1 = g_2 = a^2b\), then \(\overline{g_1\cdots g_n}\) will begin with \(a^2b\) unless the number of times the letters \(b\) and \(b^{-1}\) appear in the non-reduced word \(g_1\cdots g_n\) match for some $n\geq 2$.

    This yields the inequality
    \[\nu([a]) \ge \nu([a^2]) \ge \nu([a^2b]) \ge \mu(\lbrace a^2b\rbrace)^2 p_1(\epsilon),\]
    where \(p_1(\epsilon)\) is the probability that a random walk on \(\Z\) starting at \(2\) (the number of times \(b\) appears in \(a^2ba^2b\)) and with transition probabilities \(p(n,n-1) = p(n,n) = \epsilon\) and \(p(n,n+1) = 1-2\epsilon\) for all \(n \in \Z\), does not reach \(0\) (we use \(p(x,y)\) for the transition probability from \(x\) to \(y\)).  In particular, \(\lim\limits_{\epsilon \downarrow 0}p_1(\epsilon) = 1\) (see for example \cite[Chapter IV.19]{spitzer}).

    On the other hand, if $\overline{g_1\cdots g_n}$ begins with \(ab\) and \(g_1 = a^2b\) then \(a\) must appear exactly once more than \(a^{-1}\) in the non-reduced word \(g_1\cdots g_k\), for some $2\leq k\leq n$.
    
    Discussing according to $g_1\in\{a^{-1},b^{-1},a^{2}b\}$ yields
    \[\nu([ab]) \le \mu(\lbrace a^{-1}\rbrace) + \mu(\lbrace b^{-1}\rbrace) + \mu(\lbrace a^2b\rbrace)p_2(\epsilon),\]
    where \(p_2(\epsilon)\) is the probability that a random walk on \(\Z\) starting at \(2\) with transition probabilities \(p(n,n+2) = 1-2\epsilon\) and \(p(n,n-1) = p(n,n) = \epsilon\), will reach \(1\).   
    
    Since \(\lim\limits_{\epsilon \downarrow 0}p_2(\epsilon) = 0\), we obtain that
    \[\frac{\nu([ab])}{\nu([a])} < \frac{\nu([a^2b])}{\nu([a^2])},\]
    for small enough \(\epsilon > 0\) as claimed.
\end{example}

The hidden Markov property of \(\nu_{\Sigma^\omega_R}\) is implicit in formulas for harmonic measure in terms of matrix products in \cite{lalley}.  However, the number of states of the Markov pre-image implied by Lalley's formulas is large, and his construction would not suffice for our proof of Theorem \ref{maintheorem}.

For nearest neighbor random walks on free groups the Markov property of harmonic measure is well known and has been extended to certain types of free products and other tree-like graphs (see \cite{ledrappier-free}, \cite{gerl-woess}, \cite{picardello-woess}, \cite{mairesse}, and \cite{mairesse-matheus}).

Our result implies for example, that when \(\supp(\mu) = \lbrace ab,a^{-1},b^{-1}\rbrace\), then \(\nu_{\Sigma^\omega_R}\) is Markov.   Other examples of random walks which are not nearest-neighbor but have Markov harmonic measures can be obtained by using stopping times on a nearest-neighbor walk as in \cite{forghani}.

Since the image of a Markov measure is rarely Markov (see for example \cite{burke-rosenblatt}, \cite{kelly}, and \cite{gurvits-ledoux})) the case where \(\nu_{\Sigma^\omega_R}\) is hidden Markov but not Markov also occurs (and in some sense should be generic).

A natural context for the results of this section seems to be that of finite complete rewriting systems, i.e. finite sets of reduction rules with the property that applying available reductions in any order one eventually arrives at a unique reduced form for each word.

For example, the plain groups studied in \cite{bordenave-dubail}, \cite{mairesse}, and \cite{mairesse-matheus}, admit a finite complete rewriting system where every rule substitutes a pair of consecutive letters either with the empty string or a single letter.

There exist finite presentations of surface groups which are complete re-writing systems (see \cite{chenadec} and \cite{hermiller}).  It seems of interest, in view of Conjecture \ref{klconjecture}, to explore whether the harmonic measure on the set of infinite reduced words for these presentations have the hidden Markov property.

\subsection{Infinite non-reduced words}

As an initial step towards the proof of Theorem \ref{reducedhiddenmarkovtheorem}  we give an explicit construction for the following:
\begin{lemma}[Non-reduced infinite words are hidden markov]\label{nonreducedlemma}
 There exists a Markov chain \(y_1,y_2,\ldots\) on \(\prefix(\supp(\mu))\) such that
 \[\lim\limits_{n \to +\infty}\pi(y_1)\pi(y_2)\cdots \pi(y_n),\]
 exists, belongs to \(\Sigma^\omega\) almost surely, and has distribution \(\nu_{\Sigma^\omega}\).
\end{lemma}

\subsubsection{Weighted prefix graph}

We will construct a weighted graph out of \(\prefix(\supp(\mu))\), see Figure \ref{prefixgraphfigure} for an example. 

We say \(w'\) is a one step prefix of \(w\) and write \(w' \overset{1}{\prec} w\), if \(w = w'x\) for some \(x \in \Sigma\). More generally we write \(w' \prec w\) if \(w'\) is a prefix of \(w\).

\begin{definition}[Weighted prefix graph]
 By the prefix graph associated to \(\mu\) we mean the graph with vertex set \(\prefix(\supp(\mu))\), where a single directed edge is added from \(w'\) to \(w\) whenever \(w' \overset{1}{\prec} w\), and also from \(w\) to \(\varepsilon\) for each \(w \in \supp(\mu)\).

 To each edge of the form \(w' \overset{1}{\prec} w\) we associate the weight 
\[W(w',w) = \sum\limits_{w \prec w''}\mu(\lbrace w''\rbrace),\]
and to each edge of the form \(w \rightarrow \varepsilon\) the weight \(W(w,\varepsilon) = \mu(\lbrace w\rbrace)\).
\end{definition}

We now prove two basic properties of the weighted prefix graph which we need to prove Lemma \ref{nonreducedlemma}.

\begin{proposition}
 There exists \(N\) such that every path of length \(N\) in the prefix graph passes through \(\varepsilon\) at least once.
\end{proposition}
\begin{proof}
 Each step along a path that does not arrive at \(\varepsilon\) increases word length.  The maximal possible word length is \(\max\lbrace |w|: w \in \supp(\mu)\rbrace\) which is finite since \(\supp(\mu)\) is a finite set.
\end{proof}

\begin{proposition}[Incoming and outgoing weights are equal]\label{flowproposition}
With the above definition one has
\[\sum\limits_{y}W(y,x) = \sum\limits_{y}W(x,y),\] 
for all \(x \in \prefix(\supp(\mu))\).
\end{proposition}
\begin{proof}
If \(x = \varepsilon\) the left-hand side is \(1\) since there is a term \(\mu(\lbrace w\rbrace)\) for each \(w \in \supp(\mu)\).   The right-hand side is
\[\sum\limits_{x \in \Sigma} \sum\limits_{x \prec w} \mu(\lbrace w\rbrace),\]
which is the same sum grouped by first letter and therefore is also \(1\).

For \(x \neq \varepsilon\) one has \(x' \overset{1}{\prec}x\) for a  unique (possibly empty) word \(x' \in \prefix(\supp(\mu))\).  Hence, the left-hand side is
\[\sum\limits_{x \prec w}\mu(\lbrace w\rbrace).\]

If \(x \notin \supp(\mu)\) then the right-hand side is
\[\sum\limits_{x \overset{1}{\prec}x''} \sum\limits_{x'' \prec w} \mu(\lbrace w\rbrace),\]
which, is the sum above regrouped by the first letter after \(x\).

If \(x \in \supp(\mu)\) then the two sums differ by \(\mu(x)\) which appears in the right-hand side as \(W(x,\varepsilon)\).  This concludes the proof of the claim.
\end{proof}

\subsubsection{Proof of Lemma \ref{nonreducedlemma}}

To conclude the proof of Lemma \ref{nonreducedlemma} we consider the Markov chain on the weighted prefix graph whose transition probabilities are proportional to the given weights.

\begin{lemma}[Markov chain on the weighted prefix graph]
Let \(y_1,y_2,\ldots\) be a Markov chain on the weighted prefix graph starting at \(\varepsilon\) and with transition probabilities given by
\[P(x,y) = \frac{W(x,y)}{\sum\limits_{z}W(x,z)}.\]

Then \(\lim\limits_{n \to +\infty}\pi(y_1)\pi(y_2)\cdots \pi(y_n)\),
has distribution \(\nu_{\Sigma^\omega}\).
\end{lemma}
\begin{proof}
Let \(\tau_0 = 1\) and inductively define \(\tau_{n+1} = \min\lbrace k > \tau_n: y_{k} = \varepsilon\rbrace\) for \(n = 0,1,\ldots\).  By the strong Markov property the sequence \(y_{\tau_{1}-1}, y_{\tau_2-1},\ldots\) is i.i.d..  It suffices to show that \(y_{\tau_1 -1}\) has distribution \(\mu\) to prove the claim.

For this purpose fix \(w \in \supp(\mu)\), set \(n = |w|\), and observe using Proposition \ref{flowproposition} that
\begin{align*}
\Pof{y_{\tau_1 - 1} = w} &= P(\varepsilon,w[1])\cdots P(w[1:n-1],w)P(w,\varepsilon) 
\\ &= \frac{W(w,\varepsilon)}{\sum\limits_{x \in \Sigma}W(\varepsilon,x)} = \mu(\lbrace w\rbrace).\qedhere
\end{align*}
\end{proof}

\begin{figure}[H]
\includegraphics{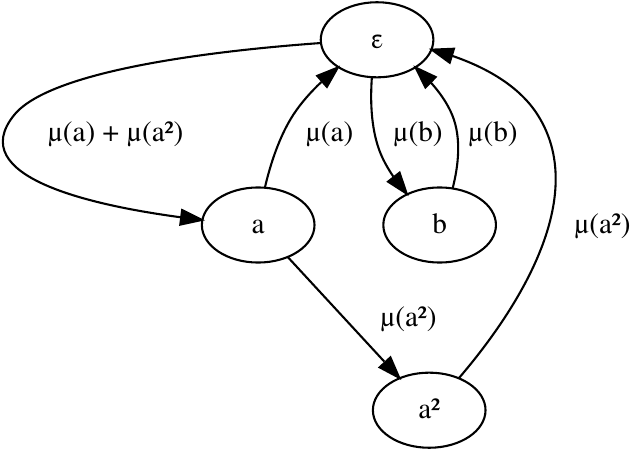}
\caption{\label{prefixgraphfigure}The weighted prefix graph associated to a probability \(\mu\) with support \(\supp(\mu) = \lbrace a, a^2, b\rbrace\).}
\end{figure}

\subsection{Proof of Theorem \ref{reducedhiddenmarkovtheorem}}

Let \(y_1,y_2,\ldots\) be as in Lemma \ref{nonreducedlemma} and define  \(S = \prefix(\supp(\mu)) \setminus \lbrace \varepsilon\rbrace\).

Given a (possibly empty) reduced word \(w\) we define the stopping time 
\[\tau_w = \min\lbrace n: \overline{\pi(y_1)\cdots \pi(y_n)} = w\rbrace.\]

Let \(f:S \to [0,1]\) be defined by
\[f(s) = \Pof{\tau_{\pi(s)} < +\infty, y_{\tau_{\pi(s)}} = s}.\]

Similarly let \(g:S \times S \to [0,1]\) be defined by
\[g(s,s') = \P^{s}\left(\tau_{\pi(s)\pi(s')} < +\infty, y_{\tau_{\pi(s)\pi(s')}} = s'\right),\]
where \(\P^s\) denotes the probability for the Markov chain conditioned to start at \(s\).   Notice that if \(\pi(s)\pi(s')\) is not reduced then \(g(s,s') = 0\).

Finally let \(h:S \to [0,1]\) be defined by
\[h(s) = \P^s\left(\pi(s) \prec \overline{\pi(y_1)\cdots \pi(y_n)}\text{ for all }n\text{ large enough}\right).\]

\begin{definition}[Reduced prefix graph]
We observe that \(g(s,s') > 0\) if and only if there exists a path \(s = s_0 \rightarrow s_1 \rightarrow \cdots \rightarrow s_n = s'\) in the weighted prefix graph such that 
\[\overline{\pi(s_0)\cdots \pi(s_n)} = \pi(s)\pi(s')\]
and furthermore
\[\overline{\pi(s_0)\cdots \pi(s_k)} \neq \pi(s)\pi(s'),\]
for all \(k < n\).

By the reduced prefix graph we mean the set \(S = \prefix(\supp(\mu)) \setminus \lbrace \varepsilon\rbrace\) with a single directed edge between \(s\) and \(s'\) when the above property holds. 
\end{definition}

The following result concludes the proof of Theorem \ref{reducedhiddenmarkovtheorem}.

\begin{lemma}\label{markovlemma}
 Let \(r = \lim\limits_{n \to +\infty}\overline{\pi(y_1)\cdots \pi(y_n)}\) and set \(\tau_n = \min\lbrace k: \overline{\pi(y_1)\cdots \pi(y_k)} = r[1:n]\rbrace\) for \(n = 1,2,\ldots\).
 
 Then the sequence \(x_1 = y_{\tau_1},x_2 = y_{\tau_2},\ldots\) is a Markov chain and furthermore
 \begin{equation}\label{markovequation}
 \Pof{x_1 = s_1,\ldots, x_n = s_n} = f(s_1)g(s_1,s_2)\cdots g(s_{n-1},s_n)h(s_n),
 \end{equation}
 for all \(s_1,\ldots, s_n \in S\).
\end{lemma}
\begin{proof}
It suffices to prove equation \ref{markovequation}, since this implies that \(x_1,\ldots, x_n\) is Markov with initial distribution \(p(s) = f(s)h(s)\) and transtion probabilities \(P(s,s') = h(s)^{-1}g(s,s')h(s')\).

For this purpose we let \(w = \pi(s_1)\cdots \pi(s_n)\) and write
\begin{equation} 
\Pof{x_1 = s_1,\ldots, x_n = s_n} = \Pof{F_1 \cap \cdots F_n \cap G} = \Pof{F_n \cap G},
\end{equation}
where \(F_i = \lbrace \tau_{w[1:k]} < +\infty, y_{\tau_{w[1:k]}} = s_k\text{ for all }k \le i\rbrace\) and \(G = \lbrace r\text{ reduced word}: w \prec r\rbrace\).

We observe that conditioned on \(F_n\) the event \(G\) is equivalent to
\[\left\lbrace\pi(s_n) \prec \overline{\pi(y_{\tau_w})\cdots \pi(y_{k + \tau_w})}\text{ for all }k\text{ large enough}\right\rbrace.\]

Therefore, by the strong Markov property we obtain 
\begin{equation}\Pof{F_n \cap G} = \Pof{F_n}\Pof{G | F_n} = \Pof{F_n}h(s_n).
\end{equation}

For each \(i = 1,\ldots,n\) we observe that conditioned on \(F_i\) the event \(F_{i+1}\) is equivalent to there being some finite \(k\) such that \(\overline{\pi(y_{\tau_{w[1:i]}})\cdots \pi(y_{k + \tau_{w[1:i]}})} = \pi(s_i)\pi(s_{i+1})\) and the minimal \(k\) with this property satisfying \(y_{k + \tau_{w[1:i]}} = s_{i+1}\).

Once again by the strong Markov property we obtain 
\begin{equation}
\Pof{F_{i+1}} = \Pof{F_i}\Pof{F_{i+1}|F_i} = \Pof{F_i}g(s_i,s_{i+1}).
\end{equation}

Hence we have 
\begin{align*}
 \Pof{x_1 = s_1,\ldots, x_n = s_n} &= \Pof{F_n \cap G}
 \\ &= \Pof{F_n}\Pof{G|F_n}
 \\ &= \Pof{F_{n-1}}\Pof{F_n|F_{n-1}}h(s_n)
 \\ &= \cdots 
 \\ &= \Pof{F_1}g(s_1,s_2)\cdots g(s_{n-1},s_n)h(s_n),
\end{align*}
the equality \(\Pof{F_1} = f(s_1)\) is direct from the definition of \(f\), which concludes the proof.
\end{proof}

\section{Symbolic coding and thermodynamic formalism}

Recall that there exists family of half-planes \(\lbrace H_x: x \in \Sigma\rbrace\) whose closures in \(\H^2 \cup \partial \H^2\) are disjoint and such that for all \(x\) one has \(\rho(x)H_y \subset H_x\) for all \(y \in \Sigma \setminus \lbrace x^{-1}\rbrace\).

We fix a basepoint \(o \in \H^2 \setminus \bigcup\limits_{x \in \Sigma}H_x\) and let \(\dist\) denote the hyperbolic distance.   See Figure \ref{schottkyfigure}.

Given a finite reduced word \(w \in \Sigma^*_R\) we denote the set of infinite reduced words having \(w\) as a prefix by
\[[w] = \lbrace w' \in \Sigma^\omega_R: w \prec w'\rbrace.\]

\begin{theorem}\label{dimdropalternativetheorem}
 In the setting of Theorem \ref{maintheorem} either \(\dim(\nu) < \dim(\Lambda)\) or there exists a constant \(C > 1\) such that
 
 \[C^{-1}\exp(-\delta \dist(o,\rho(\omega)o)) \le \nu_{\Sigma^\omega_R}\left([w]\right) \le C\exp(-\delta \dist(o,\rho(\omega)o)),\] 
 for all \(w \in \Sigma^{*}_R\), where \(\delta = \dim(\Lambda)\).
\end{theorem}

\subsection{Coding of the limit set}

We recall basic facts on the symbolic coding of \(\Lambda\), see \cite[Section 9]{lalley2}.

The mapping \(\pi_{\Lambda}:\Sigma^\omega_R \to \Lambda\) defined by
\[\pi_\Lambda(w) = \lim\limits_{n \to +\infty}\rho(w[1:n])o,\]
is a Hölder homeomorphism when \(\Lambda\) is endowed with the visual distance \(\dist_o\) based at \(o\) and \(\Sigma_R^\omega\) with the distance 
\[\dist(w,w') = \exp\left(-\min\lbrace n: w[n] \neq w'[n]\rbrace\right).\]

We define \(f:\Sigma^\omega_R \to \R\) as
\[f(w) = \log|F'(\pi_\Lambda(w))|,\]
where \(F:\Lambda \to \Lambda\) restricted to \(\overline{H_x} \cap \partial \H^2\) is \(\rho(x^{-1})\) for each \(x \in \Sigma\), and the derivative \(F'\) is taken with respect to the constant speed parametrization with respect to visual distance based at \(o\).

\subsection{Thermodynamical formalism}

The mapping \(F\) is conjugate via \(\pi_\Lambda\) to the left shift \(\sigma:\Sigma^\omega_R \to \Sigma^\omega_R\) defined by \(\sigma(w)[i] = w[i+1]\).

The Hausdorff dimension \(\delta = \dim(\Lambda)\) is characterized (see \cite[Section 4]{bowen}) as the unique value of \(s \ge 0\) such that \(P(s) = 0\) where
\[P(s) = \sup\limits_{m} \left\lbrace h(m) - s\int f(w)dm(w)\right\rbrace,\]
and the supremum is taken over all shift-invariant probability measures \(m\), and \(h(m)\) denotes topological entropy of \(m\).  There is a unique shift-invariant probability measure \(m_s\) attaining the supremum above for each $s\geq 0$.

Recall, that the Gibbs measure on \(\Sigma^\omega_R\) corresponding to a Hölder potential \(\psi:\Sigma^\omega_R \to \R\) is the unique shift-invariant probability such that for some \(C > 1\) one has
\[C^{-1}\exp\left(-nP + \sum\limits_{k = 0}^{n-1}\psi(\sigma^k(w'))\right) \le m([w]) \le C\exp\left(-nP + \sum\limits_{k = 0}^{n-1}\psi(\sigma^k(w'))\right),\]
for all \(w \in \Sigma^*_R\) and all \(w' \in [w]\), where \(P\) is the pressure of \(\psi\), which is equal to \(0\) in the case of \(m_\delta\).

Observe that \(m_\delta\) is the Gibbs measure corresponding to the potential \(-\delta f\).  We will show, using standard results in thermodynamical formalism, that \(m_\delta\) maximizes the dimension of \(\pi_*m\) among Gibbs measures.

\begin{lemma}\label{ergodicdimlemma}
If \(m\) is a Gibbs measure on \(\Sigma^\omega_R\) then either \(m = m_\delta\) or \(\dim(\pi_*m) < \dim(\Lambda)\).
\end{lemma}
\begin{proof}
   From \cite[Theorem 8.1.4]{urbanksi} one has that \(\pi_*m\) is exact dimensional with dimension
   \[\dim(\pi_*m) = \frac{h(m)}{\int f(w)dm(w)}.\]

   Since \(m_\delta\) is the unique shift-invariant measure attaining the supremum \(P(\delta) = 0\), for all \(m \neq m_\delta\) one has
   \[h(m) - \delta \int f(w) dm(w) < P(\delta) = 0.\]
   Hence, \(\dim(\pi_*m) < \delta\) as claimed.
\end{proof}

For \(\xi \in \partial \H^2\) we denote by \(b_\xi:\H^2 \times \H^2 \to \R\) the Busemann function defined by
\[b_\xi(x,y) = \lim\limits_{t \to +\infty}\dist(\alpha(t),x) - \dist(\alpha(t),y),\]
where \(\alpha:[0,+\infty) \to \H^2\) is any geodesic ray ending at \(\xi\) parametrized by arclength (see \cite[Theorem 1.18]{dalbo}).

Directly from the definition it follows that \(b_\xi(x,z) = b_\xi(x,y) + b_\xi(y,z)\) for all \(x,y,z \in \H^2\).  Furthermore one has \(|b_\xi(x,y)| \le \dist(x,y)\).

We will need the following additional property for what follows:
\begin{lemma}\label{horofunctionlemma}
For all \(g \in G\) one has \(|(g^{-1})'(\xi)| =\exp(b_\xi(o,g o))\) where \(g'\) is the derivative with respect to the constant speed parametrization with respect to the visual distance on \(\partial \H^2\) based at \(o\).
\end{lemma}
\begin{proof}
The proof is by direct calculation in the Poincaré disk model where \(\H^2\) is identified with \(\D = \lbrace z \in \C: |z| < 1\rbrace\) and \(\partial \H^2\) with the boundary circle \(\partial \D\).  Setting \(o = 0\), the parametrization of \(\partial \D\) is by Euclidean arclength.

We fix an orientation preserving isometry of \(g\) which in this model is of the form
\[g(z) = \frac{az + b}{\overline{b}z + \overline{a}},\]
with \(|a|^2 - |b|^2 = -1\).   By direct calculation
\[g^{-1}(z) = \frac{\overline{a}z - b}{-\overline{b}z + a}.\]

We calculate \(|(g^{-1})'(\xi)| = 1/|a -\overline{b}\xi|^2\) while one has (see \cite[pg. 273]{bridson-haefliger} and \(b_{\xi,x}\) defined there corresponds to \(-b_\xi(o,x)\) in our notation)
\[\exp(b_\xi(o,g(o))) = \frac{1 - |g(0)|^2}{|g(0)-\xi|^2} = \frac{1 - |b|^2/|\overline{a}|^2}{|b/\overline{a} - \xi|^2} = 1/|b -\overline{a}\xi|^2.\]

Since \(|\xi| = 1\) we have
\[|a-\overline{b}\xi| = |\overline{\overline{b}\xi-a}| = |b\overline{\xi}-\overline{a}| = |\overline{\xi}(b - \overline{a}\xi)| = |b-\overline{a}\xi|. \qedhere\]
\end{proof}

We now establish the estimates for \(m_\delta\) needed for Theorem \ref{dimdropalternativetheorem}.

\begin{lemma}\label{bowenmeasurelemma}
  There exists \(C > 1\) such that
  \[C^{-1}\exp(-\delta \dist(o,\rho(w)o)) \le m_\delta([w]) \le C\exp(-\delta\dist(o,\rho(w)o)),\]
  for all \(w \in \Sigma^*_R\).
\end{lemma}
\begin{proof}
From the Gibbs property of \(m_\delta\) we obtain
\[m_\delta([w]) \asymp |(F^{|w|})'(\pi_{\Lambda}(\tilde{w}))|^{-\delta} = |(\rho(w)^{-1})'\xi|^{-\delta}\]
where \(F^k\) denotes the \(k\)-th iterate of \(F\), \(\asymp\) means up to a multiplicative constant independent of \(w\), \(\tilde{w} \in [w]\), and \(\xi = \pi_\Lambda(\tilde{w})\).

By Lemma \ref{horofunctionlemma} we have
 \[|(\rho(w)^{-1})'\xi|^{-\delta} = \exp(-\delta b_\xi(o,\rho(w)o)).\]

Given \(x \in \H^2\) and \(r > 0\) we will use \(B(x,r) = \lbrace y: \dist(x,y) < r\rbrace\) to denote the open ball of radius \(r\) centered at \(x\) for the hyperbolic metric.    

We will now work with the Poincaré disk model of \(\H^2\) setting \(o = 0\), we will need the following two facts (see \cite[Chapter 4]{anderson}):
\begin{enumerate}
    \item The hyperbolic open ball \(B(o,r)\) is an Euclidean open disk centered at \(0\) for all \(r > 0\).
    \item Given \(x,y \in \Sigma\) the complete geodesics joining \(H_x\) and \(H_y\) are either Eucliean diameters of the disk, or Euclidean circle arcs perpendicular to the boundary circle \(S^1\).
\end{enumerate}   

Since \(\overline{H_x}\cap S^1\) and \(\overline{H_y}\cap \R\) are disjoint closed arcs in the boundary circle \(S^1\), we obtain that if \(R > 0\) is sufficiently large then all geodesics joining \(H_x\) to \(H_y\) for distinct \(x,y \in \Sigma\) intersect \(B(o,R)\).    

We now consider the geodesic ray \([o,\xi)\) from \(o\) to \(\xi\) and notice that \(\rho(w)^{-1}[o,\xi) = [\rho(\omega)^{-1}o, \rho(w)^{-1}\xi)\) joins a point in \(H_{w[n]^{-1}}\) to a point in \(H_{\tilde{w}[n+1]}\).  Since \(\tilde{w} \in \Sigma^\omega_R\) we have \(w[n]^{-1} \neq \tilde{w}[n+1]\) and therefore this geodesic intersects \(B(o,R)\).

Applying the isometry \(\rho(w)\) we obtain that the geodesic ray \([o,\xi)\) intersects \(B(\rho(w)o,R)\).  It follows that
\[|b_\xi(o,\rho(w)o) - \dist(o,\rho(w)o)| \le 2R.\qedhere\]
\end{proof}

\subsection{Proof of Theorem \ref{dimdropalternativetheorem}}

From \cite[Proposition 3.2]{ledrappier} there exists a continuous positive density \(\varphi\) such that \(\varphi\nu_{\Sigma^\omega_R}\) is a Gibbs measure.

Assume that \(\dim(\nu) = \dim(\Lambda)\).  Since \(\nu = \pi_*\nu_{\Sigma^\omega_R}\) from Lemma \ref{ergodicdimlemma} it follows that \(\varphi\nu_{\Sigma^\omega_R} = m_\delta\).  In particular, since \(\varphi\) is bounded away from zero and infinity there exists \(C' > 1\) such that
\[(C')^{-1}m_{\delta}([w]) \le \nu_{\Sigma^\omega_R}([w]) \le C'm_{\delta}([w]),\]
for all reduced words \(w\).  Combining this with Lemma \ref{bowenmeasurelemma} concludes the proof.

\section{Proof of Theorem \ref{maintheorem}}

We denote the translation length of an isometry \(g\) of \(\H^2\) by
\[\ell(g) = \lim\limits_{n \to +\infty}\frac{1}{n}\dist(o,g^n o).\]

We say a word \(w \in \Sigma^*_R\) is cyclically reduced if its first and last letters are not inverses.

\begin{lemma}\label{cyclicallyreducedlemma}
  Either \(\dim(\nu) < \dim(\Lambda)\) or for all cyclically reduced \(w \in \Sigma^*_R\) one has
  \[\lim\limits_{n \to +\infty}\left(\nu_{\Sigma^\omega_R}([w^n])\right)^{\frac{1}{n}} = \exp(-\delta \ell(\rho(w))).\]
\end{lemma}
\begin{proof}
  This is a direct corollary of Theorem \ref{dimdropalternativetheorem}.
\end{proof}

\begin{lemma}\label{additivelemma}
  Either \(\dim(\nu) < \dim(\Lambda)\) or for all cyclically reduced \(aw_1, aw_2 \in \Sigma^*_R\) one has
  \[\ell(\rho(aw_1aw_2)) = \ell(\rho(aw_1)) + \ell(\rho(aw_2)).\]
\end{lemma}
\begin{proof}
  
  Given \(w \in \Sigma^*_R\) we denote by \(A(w)\) the set of \(w\)-paths, which is to say \(\alpha \in A(w)\) if \(\alpha =  (s_0 \rightarrow \cdots \rightarrow s_n)\) is a path in the reduced prefix graph such that \(\pi(s_0)\cdots \pi(s_{n}) = w\).   For such paths we write \(g(\alpha) = g(s_0,s_1)\cdots g(s_{n-1},s_n)\), and \(g_+(\alpha) = g(\alpha)h(s_n)\).  
  
  By hypothesis there is a unique element \(s_a \in \prefix(\supp(\mu))\) such that \(\pi(s_a) = a\).   This implies that if \(\alpha \in A(w)\) and \(w[i] = a\) then the \(i\)-th vertex visited by \(\alpha\) is \(s_a\).

  Using this observation and Lemma \ref{markovlemma} we calculate for \(aw\) cyclically reduced 
  \begin{align*}
 \nu_{\Sigma^\omega_R}([(aw)^n]) = f(s_a)\left(\sum\limits_{\alpha \in A(awa)}g(\alpha)\right)^{n-1}\sum\limits_{\alpha \in A(aw)}g_+(\alpha).
  \end{align*}
  
  By Lemma \ref{cyclicallyreducedlemma} this implies
  \[\exp(-\delta \ell(\rho(aw))) = \sum\limits_{\alpha \in A(awa)}g(\alpha),\]
  whenever \(aw\) is cyclically reduced.
  
  In particular we have
  \begin{align*}\exp(-\delta \ell(\rho(aw_1aw_2))) &= \sum\limits_{\alpha \in A(aw_1aw_2a)}g(\alpha)
  \\ &= \sum\limits_{\alpha \in A(aw_1a)}g(\alpha) \sum\limits_{\alpha \in A(aw_2a)}g(\alpha)
  \\ &= \exp(-\delta\ell(\rho(aw_1)))\exp(-\delta\ell(\rho(aw_2))),
  \end{align*}
  which concludes the proof.
\end{proof}

If it were the case that \(\dim(\nu) = \dim(\Lambda)\), by Lemma \ref{additivelemma} setting \(w = aa\) and \(w' = ab\) one would have \(\ell(\rho(ww')) = \ell(\rho(w)) + \ell(\rho(w'))\).

Notice that \(\rho(w)\) is a translation along a geodesic beginning in \(H_{a^{-1}}\) and ending in \(H_a\) while the axis of \(\rho(w')\) begins in \(H_{b^{-1}}\) and ends in \(H_a\).  

If the axes intersect then \(\ell(\rho(ww')) < \ell(\rho(w)) + \ell(\rho(w'))\).   If they do not, they must be at a positive distance, and since both translations are toward the endpoint in \(H_a\) we would have \(\ell(\rho(ww')) > \ell(\rho(w)) + \ell(\rho(w'))\) (see for example \cite[Theorem 2.1]{furman}).   

In both cases \(\ell(\rho(ww')) \neq \ell(\rho(w)) + \ell(\rho(w'))\), from which we deduce that \(\dim(\nu) < \dim(\Lambda)\).

\section{Acknowledgments}
The authors would like to thank Françoise Dal'bo, Sébastien Gouëzel and François Ledrappier for several helpful conversations. And Steve Lalley, Jean Mairesse, Pablo Shmerkin, and Wolfgang Woess for directing us towards relevant literature. The authors also thank the anonymous referee for the careful reading of the manuscript and the constructive remarks.
\bibliographystyle{alpha}
\bibliography{biblio}

\end{document}